\documentclass[11pt]{amsart}
\usepackage{geometry}                
\usepackage{graphicx}
\usepackage{amssymb}
\usepackage{epstopdf}
\newtheorem{thm}{Theorem}[section]
\newtheorem{lem}[thm]{Lemma}
\newtheorem{cor}[thm]{Corollary}

\newtheorem{clm}[thm]{Claim}

\theoremstyle{definition}
\newtheorem{defn}{Definition}[section]

\newtheorem{ques}[defn]{Question}
\newtheorem{note}[defn]{Note}

\theoremstyle{remark}
\newtheorem{rmk}{Remark}[section]

\def\square{\hfill${\vcenter{\vbox{\hrule height.4pt \hbox{\vrule
width.4pt height7pt \kern7pt \vrule width.4pt} \hrule height.4pt}}}$}

\title{An upper bound on distance degenerate handle additions }
\author{Yanqing Zou}
\date{April 14th, 2017}                                           
\thanks{This work was partially supported by NSFC No. 11601065.}

\begin{document}
\begin{abstract}
We prove that for any distance at least 3 Heegaard splitting and  a boundary component $F$, there is  a  diameter finite ball  in the curve complex $\mathcal {C}(F)$  so that it contains all distance degenerate curves or slopes in $F$.
\end{abstract}
\maketitle

\vspace*{0.5cm} {\bf Keywords}: Heegaard Distance, Curve Complex, Handle Addition.\vspace*{0.5cm}


\section{Introduction}
\label{sec1}
Let $M$ be a compact orientable 3-manifold with a boundary component $F$. Then it admits a Heegaard splitting $V\cup_{S}W$, where $F\subset V$. Here $S$ is  $\partial_{+}V$ (resp. $\partial_{+}W$) and $\partial_{-}V=\partial V-S$ (resp. $\partial_{-}W=\partial W-S$). 

Let $r$ be a slope or an essential simple closed curve in $F$.  Then  a Dehn filling or a handle addition along $r$ on $M$ produces a 3-manifold $M(r)$.  Since $V\cup_{S} W$ is a Heegaard splitting of $M$,  $M(r)$ admits a Heegaard splitting $V(r)\cup_{S}W$,  where $V(r)$ is  obtained  from attaching a 2-handle along $r$ on $V$ and capping a possible 2-sphere by a 3-ball.

There  is a long story on studying handle additions or Dehn fillings on a 3-manifold.  Lickorish \cite{lickorish} proved that every closed orientable 3-manifold is a Dehn surgery along some link or knot in $S^{3}$.  For any knot $K$ in $S^{3}$,  Gordon and Luecke \cite{GL}  proved that  only the trivial Dehn surgery produces  $S^{3}$.  In general,  Culler,  Gordon, Luecke and Shalen introduced a cyclic surgery theorem, see \cite{CGLS}.  One of its corollaries is that  only integer surgery on a non torus knot can produce a cyclic fundamental group.  

Given a knot $K$ in $S^{3}$,  it is either prime or a connected sum of some prime knots. Let $\eta(K)$ be the regular neighborhood of $K$ in $S^{3}$. If $K$ is prime,  then it is either hyperbolic, i.e.,  $E(K)=S^{3}\setminus \eta(K)$ admits a complete hyperbolic metric, or a torus knot or a satellite knot.  But if $K$ is a connected sum of some knots, then there is a properly embedded essential annulus in $E(K)$. In this case, by Thurston's Haken hyperbolic theorem, $E(K)$ admits no hyperbolic structure.  Thus with respect to the geometry of $E(K)$,  hyperbolic knots are mostly concerned. For a hyperbolic knot $K$, Thurston \cite{thurston} proved that all but finitely many Dehn fillings on $E(K)$ produce hyperbolic 3-manifolds. It was conjectured by Gordon \cite{gordon98} that (1) there are at most 10 non hyperbolic Dehn fillings (there are 10 non hyperbolic Dehn fillings for the figure eight knot); (2) the intersection number of any two non hyperbolic slopes is at most 8. Recently,  Agol \cite{agol1} proved that for all but finitely many one cusped hyperbolic 3-manifolds, the intersection number is 5 while there are at most 8 non hyperbolic Dehn fillings.  Later Lackenby and Meyerhoff \cite{LM} proved this conjecture completely. Moreover, by  Thurston's Geometrization conjecture \cite{thurston} proved by Perelman \cite{Perelman01, Perelman02, Perelman03},  except a small Seifert fiber space, every closed orientable non hyperbolic 3-manifold is either reducible or toroidal. So if we consider the reducible Dehn fillings on $E(K)$ , i.e., the  resulted 3-manifold is reducible,   then the number 10 is reduced to 2, see \cite{GL2}.  For more cases,  see \cite{gordon}.

It is natural to extend these Dehn fillings results  into a handle addition on a hyperbolic  3-manifold. Then we consider  a hyperbolic 3-manifold $M$ with a totally geodesic boundary component  $F$.  Before stating some  results about  handle additions on $F$, we introduce a definition. An essential simple closed curve $r\subset F$ is called a non degenerate curve if $M(r)$ is also hyperbolic. Otherwise, it is degenerate.  Then if $r$ is  degenerate in $F$, by Thurston's Haken hyperbolic theorem,  $M(r)$ is  reducible, or boundary reducible, or annular  or toroidal. So to figure out all non degenerate curves in $F$,  it is sufficient to give a classification of all degenerate curves from the topology of $M(r)$.   Scharlemann and Wu \cite{sw} studied all those degenerate curves on $F$ and proved that there are finitely many basic degenerate curves in $F$ so that either each degenerate curve   is basic or it bounds a pair of pants with a basic degenerate curve. It means that for most of all essential simple closed curves, $M(r)$ is hyperbolic.   Unfortunately there is no upper bound on their intersection numbers among all  degenerate curves in $F$, for example, two complicated intersecting  degenerate curves  with respect to a same basic degenerate curve.  However, if we consider  two separating reducible handle addition curves,  then their intersection number is at most 2,  see \cite{QZ}.  Meanwhile, Lackenby \cite{La}  introduced a handlebody addition along $F$  and proved that  there is an upper bound on all non hyperbolic handlebody additions.

It is known that every Heegaard surface of $M$ is also a Heegaard surface of $M(r)$.  The properties of  a Heegaard splitting of $M$ under a Dehn filling or a handle addition  are concerned, such as the minimal genus, Heegaard distance.  It is not hard to see that the minimal Heegaard genus of $M(r)$ is not larger than $M$'s.  Then it is interesting to know that when they have the same  minimal Heegaard genera.  There are some results as follows: Rieck \cite{rieck01}  proved that for most of all $r$ in $F$, the minimal Heegaard genus of $M(r)$ is  at most  one less then  $M$'s; Moriah and Sedgwick \cite{MS} proved that for all but finitely many curves in $F$,  $M(r)$ has the same genus as $M$; Li \cite{li5} proved that  if the gluing map of a handlebody addition is sufficiently complicated,  then  the resulted 3-manifold has the same minimal Heegaard genus as $M$. 

Hempel \cite{hempel} introduced the Heegaard distance for studying a Heegaard splitting. More precisely, let $\{\alpha_{0},...,\alpha_{n}\}$ be a collection of essential simple closed curves in $S$ so that for any $1\leq i\leq n$, $\alpha_{i}$ is disjoint from $\alpha_{i-1}$. Then for a Heegaard splitting $V\cup_{S}W$, the Heegaard distance $d(V,W)$ is the minimum of all $n$ so that $\alpha_{0}$ (resp. $\alpha_{n}$) bounds a disk in $V$ (resp. $W$). Since each essential disk in $V$ is also an essential disk of $V(r)$,   $d(V(r), W)\leq d(V,W)$.  So there is a question. 

\begin{ques}
\label{question1.1}
Is $d(V(r),W)=d(V,W)$?
\end{ques}

Unfortunately,   for $M=E(K)$, some high distance knot $K$ ( see Minsky, Moriah and Schleimer \cite{MMS}),    if $r$ is the meridian,  then $M(r)$ is $S^{3}$.  By Waldhausen theorem \cite{Waldhausen1},  every genus at least 2 Heegaard splitting of $S^{3}$ is stabilized and hence has distance 0. So the answer to Question \ref{question1.1} is no. 

However, by those results of Hempel\cite{hempel}, Hartshorn \cite{ha} and Scharlemann \cite{sch01}, if $M$ admits  a distance at least 3 Heegaard splitting, then it is irreducible, boundary irreducible,  atoroidal and anannular.  Then by Thurston's Haken hyperbolic theorem,  it is hyperbolic.  Compared with Schlarlemann and Wu's hyperbolic handle addition theorem,  it was conjectured by Ma and Qiu  \cite{MQ} that $d(V(r), W)=d(V,W)$ for most of  all curves.

To quote such an exceptional curve as the meridian of a knot,  it is proper to introduce the definition of a distance degenerate curve.  We say  $r$ is a distance degenerate curve in $F$  if $d(V(r), W)$ is less than $d(V,W)$. Furthermore, attaching 2-handle  to a 3-manifold along a distance degenerate curve  is called a distance degenerate handle addition. By standard techniques, there is a theorem as follows.

\begin{thm}
\label{theorem1.1} If the Heegaard distance of  $V\cup_{S}W$ is at least 3, then there  are an essential simple closed curve $c\subset F$ and a real number $\mathcal {R}>0$ so that 
for any distance degenerate curve $r$ in $F$, $d_{\mathcal{C}(F)}(c,r)< \mathcal{R}$. 
\end{thm}
\begin{note}
There is a precise description of $R$ in Page 22, Section 4.
\end{note}
It is not hard to see that for a distance at least 3 Heegaard splitting, every degenerate curves  in Schlarmann and Wu's hyperbolic  addition  theorem is  also a distance degenerate curve. So Theorem \ref{theorem1.1} gives a bound for all those degenerate  curves in the curve complex.

\begin{rmk}
If $M$ is $T^{2}\times I$, then it admits an unique strongly irreducible Heegaard splitting. It is known that  for any slope $r\subset \partial M$, $M(r)$ is a solid torus.  So every Heegaard splitting of $M(r)$ is weakly reducible and hence has distance at most 1.  It means that for any slope $r\subset \partial M$,  $r$ is a distance degenerate slope with respect to this strongly irreducible Heegaard splitting.
\end{rmk} 

\begin{rmk}
Lustig and Moriah \cite{LM} proved that  there is a measure defined on the curve complex $\mathcal {C}(F)$ so that for any $\mathcal {R}> 0$ and any essential simple closed curve $c$, the measure of a $\mathcal {R}$-ball of $c$ is 0.  Under this circumstance, for almost all  choices of $r$ in $F$, $d(V(r),W)=d(V,W)$ .
\end{rmk}

\begin{rmk}
Let $c_{1}$ and $c_{2}$ be two separating essential simple closed curves in $S$.  Suppose that  $d_{\mathcal {C}(S)}(c_{1}, c_{2})=l\geq 3$.  Then attaching two 2-handles along $c_{1}$ and $c_{2}$ from two different sides of $S$ produces a Heegaard splitting, denoted by $V\cup _{S}W$. Since $V$ (resp. $W$) has only one essential disk up to isotopy,  the distance of $V\cup_{S}W$ is equal to $l$.  Then by Theorem \ref{theorem1.1},  we can attach  2-handles to its boundary and some 3-balls so that  $V$ (resp $W$) is changed into a handlebody $H_{1}$ (resp. $H_{2}$) and furthermore $d(H_{1}, H_{2})=l$,  see also in \cite{IJK, QZG, ZQZ}.
\end{rmk}

\begin{rmk}
If  $V\cup_{S}W$ is genus two Heegaard splitting,  Ma, Qiu and Zou \cite{MQZ} proved the main theorem by a different method.  Meanwhile, there is a result proved  by Liang, Lei and Li \cite{LLL}, which says that if the Heegaard splitting is locally complicated,  then there is a bound for all distance degenerate curves in $\mathcal {C}(F)$. 
\end{rmk} 

\begin{rmk}
If $d(V,W)\geq 2g(S)$,  then by Scharlemann and Tomova's result \cite{ST}, $V\cup W$ is a minimal Heegaard splitting. Then by the proof of Theorem \ref{theorem1.1},  we can attaching a handlebody $H$ along distance non degenerate slopes or curves in $F$  so that $V(H)\cup_{S}W$ is still a minimal Heegaard splitting. So it gives a description of Li's sufficiently complicated gluing map between a handlebody and $M$ in \cite{li5}.
\end{rmk}
We call a knot $K$ in $S^{3}$ a high distance knot if  $E(K)$ admits a distance at least 3 Heegaard splitting. 
It is known that for any knot $K\subset S^{3}$ and any distance at least 3 Heegaard splitting of $E(K)$,  the meridian  is  a distance degenerate slope.  For  $M(r)$ is $S^{3}$ and by Waldhausen theorem \cite{Waldhausen1},  every genus  at least 2 Heegaard splitting is stabilized and thus has distance 0.  Then we choose the meridian  as  the center  among all distance degenerate slopes of $E(K)$'s all distance at least 3 Heegaard splittings.  So Theorem \ref{theorem1.1} is updated into the following corollary.

\begin{cor}
\label{theorem1.2}
For any high distance knot $K\subset S^{3}$,  there is a $R_{K}$-ball of the meridian  in $\mathcal {C}[\partial E(K)]$  so that  it contains all distance degenerate slopes of $E(K)$'s all distance at least 3 Heegaard splittings.
\end{cor}

This paper is organized as follows.  We introduce  some results of a compression body in Section \ref{sec2} and some lemmas of  the curve complex  in Section \ref{sec3} .  Then we give  proofs of Theorem \ref{theorem1.1} and Corollary \ref{theorem1.2}  in Section \ref{sec4}.

 {\bf Acknowledgement.} We would like to thank Jiming Ma and Ruifeng Qiu for pointing out Question \ref{question1.1},  thank Ruifeng Qiu for many discussions and pointing out  some mistakes  in our early draft. 

\section{Subsurface projection of the disk complex}
\label{sec2}
Let $S$ be a closed orientable genus at least 2 surface.  Harvey\cite{h81}  introduced the curve complex on $S$, denoted by $\mathcal {C}(S)$, as follows. The vertices consist of all isotopy classes of essential, i.e., incompressible and non peripheral, simple closed curves in $S$.  A $k$-simplex is a collection of $k+1$ vertices which are presented by  pairwise non isotopy and disjoint  essential simple closed curves. 

Let $F$ be a compact orientable surface. If $F$ is  an  at most once punctured torus,  then  $\mathcal {C}(F)$ is defined as follows. The vertices consist of all isotopy classes of essential, i.e., incompressible and non peripheral, simple closed curves in $F$. A $k$-simplex is the collection of $k+1$ vertices which are presented by  pairwise non isotopy and intersecting one point  essential simple closed curves.   If $F$ is a fourth punctured 2-sphere, then the definition of $\mathcal {C}(F)$ is slightly different, which is  defined as follows. The vertices consist of all isotopy classes of essential, i.e., incompressible and non peripheral, simple closed curves in $F$. A $k$-simplex is the collection of $k+1$ vertices which are presented by  pairwise non isotopy and intersecting twice essential simple closed curves. In general, if $\chi(F)\leq -2$, the definition of  $\mathcal {C}(F)$ is similar to  $\mathcal {C}(S)$. 

 It is assumed that the length of an edge in $\mathcal {C}(F)$ is 1.  Then for any two vertices $\alpha$ and $\beta$,  $d_{\mathcal {C}(F)}(\alpha, \beta)$ is defined to be the minimum of lengths of paths from $\alpha$ to $\beta$ in $\mathcal {C}(F)$.  So if $\alpha$ is disjoint from but not isotopic to $\beta$,  then there is an edge between them and so $d_{\mathcal {C}(F)}(\alpha, \beta)=1$. What if  $\alpha$ intersects $\beta$?

\begin{lem}
\label{lemma2.0}
If $\alpha$ intersects $\beta$ in $\mathcal {N}$ points up to isotopy,  then $d_{\mathcal {C}(F)}(\alpha, \beta)\leq 2\log_{2}2\mathcal{N}+1$.
\end{lem}
\begin{proof}
See the proof of Lemma  1.21  in \cite{schleimer2}.
\end{proof}
Suppose $F\subset S$ is an essential subsurface, i.e., $\partial F$ is incompressible in $S$.  Masur and Minsky \cite{mm00} introduced the subsurface projection from $\mathcal {C}(S)$ to $\mathcal {C}(F)$ as follows. Let $\alpha$ be a vertex in $\mathcal {C}(S)$, where $\alpha\cap F\neq \emptyset$ up to isotopy. Then either $\alpha$ is an essential simple closed curve in $F$ or $\alpha\cap \partial F\neq \emptyset$. In the former  case, 
 the subsurface projection of $\alpha$, denoted by $\pi_{F}(\alpha)$, is $\alpha$.  In the later case,  $\alpha$ intersects $\partial F$ efficiently, i.e., there is no bigon bounded by them in $S$. Let $a$ be an arbitrary one  arc of $\alpha\cap F$. Then $\pi_{F}(\alpha)$ is  defined to be an arbitrary one essential simple closed curve of $\partial \overline{N(a\cup \partial F)}$ in $F$.  Under the definition of  the subsurface projection, if  two disjoint curves $\alpha$ and $\beta$ both cut $F$, i.e., neither $\pi_{F}(\alpha)$ nor $\pi_{F}(\beta)$ is an empty set,  then $d_{\mathcal {C}(F)}(
\pi_{F}(\alpha), \pi_{F}(\beta))\leq 2$.  

If $\partial F$ is not connected,   some essential simple closed curve of $F$ cutting out  a planar surface while  some one  doesn't. To distinguish these two kinds of essential simple closed curves in $F$,   we introduce the  definition of a strongly essential curve,  see also in \cite{ZDQG}.

\begin{defn}
\label{defintion2.1}
An essential simple closed curve $C\subset F$ is strongly essential if $C$ doesn't  cut out a planar surface in $F$. 
\end{defn}

Similarly, for a properly embedded essential arc $a\subset F$,  $a$ is strongly essential if $\pi_{F}(a)$ is strongly essential in $F$. Otherwise, it is not strongly essential in $F$.   

If $S= \partial_{+} V$,  then there is a disk complex defined on $S$, denoted by $\mathcal{D}(S)$. The vertices consist of all isotopy classes of boundary curves of essential disk of $V$. A $k$-simplex is the collection of $k+1$ vertices which are  pairwise non isotopy and disjoint.  It is not hard to see that $\mathcal{D}(S)$ is a subcomplex of $\mathcal {C}(S)$.  Thus for an essential subsurface $F\subset S$, there is a subsurface projection from $\mathcal {D}(S)$ to $\mathcal {C}(F)$. Throughout  the finer structure of $\mathcal {D}(S)$,  Li\cite{li3}, Masur and Schleimer \cite{ms} proved that if $\partial F$ is disk-busting, i.e., it intersects the boundary of  every essential disk nonempty, then there is a bound on the diameter of subsurface projection of the  disk complex  for almost all cases.  More precisely, it is written as follows.
\begin{lem}
\label{lemma2.1}
 Let $F$ be a connected subsurface of $S$ so that  each component of $\partial F$ is
disk-busting. Then 
\begin{enumerate}
\item either $V$ is an $I$-bundle over a compact surface, $F$ is a component of the horizontal boundary of this $I$-bundle, and the vertical boundary of this $I$-bundle is a single annulus, or
\item $\pi_{F}(\mathcal{D}(S))$ has diameter at most 12 in $\mathcal {C}(F)$.
\end{enumerate}
\end{lem}
\begin{note}
In Lemma \ref{lemma2.1},  if $V$ is a twisted I-bundle of $F$,  then the vertical boundary of this $I$-bundle is non separating. 
\end{note}

Let $N=\min\{\mid \partial D\cap \partial F\mid D:~ an~ essential~ disk~of~V\}$.  Suppose $D$ realizes the minimum $N$.  
Then there is a more interesting result in Li's proof of Lemma \ref{lemma2.1}.

\begin{lem}
\label{lemma2.2}
Let $V$, $S$ and $F$  be the same ones in Lemma \ref{lemma2.1}. If $N>4$, then for any  essential disk $E$ of $V$ with $\partial E\subset S$,  there is a component of $\partial E\cap F$  disjoint from a component of $\partial D\cap F$. 
 \end{lem}

\begin{proof}  
Suppose the conclusion is false. Then each component of $\partial D\cap F$ intersects every component of $\partial E\cap F$ nontrivially.  In Li's proof of Lemma 3.4 \cite{li3}, for $D$ and $E$, it is assumed that there is no cycle in their intersection. Then there is an outermost disk $\Delta$ in $E$ bounded by an arc $\delta \subset D\cap E$ and an arc $\delta^{'}\subset \partial E$.   Moreover, for  $\Delta$, there are only two types: a triangle or a quadrilateral.  Then Li \cite{li3} proved that for both of these two cases, there is a new disk $D_{1}$  so that $\mid\partial D_{1}\cap \partial F\mid < \mid \partial D\cap \partial F\mid$. But it contradicts the choice of $D$. 
\end{proof}

\section{The SE-position of an essential disk} 
\label{sec3}
Let $V$ be a nontrivial compression body with $F \subset \partial_{-}V$. 

If $\partial _{-}V =F$, then there are finitely many  disjoint and pariwise non isotopy essential disks $\cup_{i=1}^{s} B_{i}\subset V$ satisfying $F$-condition,  i.e.,  their complement in $V$ is $F\times I$. So there is a subsurface  $S_{F}=\overline{ S-\cup_{i=1}^{s} B_{i}}$  in $F\times I$, see Figure 3.1.

\begin{figure}[!htbp]
\label{figure3.1}
\includegraphics[totalheight=3cm]{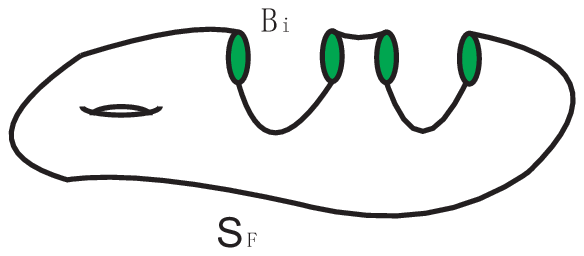}
\begin{center}
Figure 3.1.
\end{center}
\end{figure}

Let $r$ be an essential simple closed curve in $F$. Attaching a 2-handle along $r$ on $V$ (capping a possible 2-sphere by a 3-ball) produces a new compression body or handlebody, denoted by $V(r)$.  It is not hard to see that there are at least one more essential disks in $V(r)$ than $V$, for example, an essential disk $D$ containing $r$. Since each $B_{i} $ is  also an essential disk in $V(r)$,  it is interesting to know  how they intersect.

It is assumed that $D$ and $\cup_{i=1}^{s} B_{i}$ are in a general position. Then they intersect in some arcs or cycles. It is known that both a compression body and a handlebody are irreducible. Then there is no cycle in their intersection up to isotopy. So  $D\cap\cup_{i=1}^{s}B_{i}$ consists of finitely many arcs.  If $D\cap \cup_{i=1}^{s} B_{i}=\emptyset$, then $\partial D$ is strongly essential in $S_{F}$. For if not,  then $\partial D$ cuts out a planar surface in $S_{F}$. So $\partial D$ is a band sum of some components of $\partial S_{F}$.  Since $\partial S_{F}$ consists of  some essential disks' boundary  curve in $V$,  $D$ is a band sum of some essential disks in $V$. Therefore $D$ is an essential disk in $V$. It contradicts the fact that $D$ is not in $V$.  If $D\cap \cup_{i=1}^{s} B_{i}\neq \emptyset$,  then there is an outermost disk in $D$ so that it is bounded by an arc $\gamma \subset \partial D$ and an arc of $D\cap B_{i}$ for some $1\leq i \leq s$.  In this case,  $\gamma$ is either  strongly essential in $S_{F}$ or not.  If $\gamma$ is not strongly essential in $S_{F}$, then there is a boundary compression on $B_{i}$ along this outermost disk so that it produces  two essential disks $B_{i}^{1}$ and $B_{i}^{2}$.  By a standard argument,  there is at least one of them, says $B_{i}^{1}$,  so that $\{B_{1}, ... , B_{i-1}, B_{i}^{1}, B_{i+1},...B_{s}\}$ satisfies the $F$-condition. Then we consider the intersection between $D$ and  $\{B_{1}, ... , B_{i-1}, B_{i}^{1}, B_{i+1},...B_{s}\}$,  which has a fewer number intersection arcs.  If neither $D$ is disjoint from all these disks nor there is a strongly essential outermost disk in $D$,  then there is a boundary compression on these new disks again.  Since there are only  finitely many arcs between $D$ and $\cup_{i=1}^{s}B_{i}$,    finally either $D$ is strongly essential and disjoint from all these disks or there is a strongly essential outermost disk in $D$.  Then we say $D$ is in SE-position with respect to $\cup_{i=1}^{s} B_{i}$.
In general, for an essential disk $D$ in $V(r)$ but not in $V$,  there is also a SE-position for it in  $V(r)$ as follows. 

Let $\mathcal{A}$ be a collection of  all non separating spanning annuli in $V$,  where  none of their boundary curves lies in $F$. Let $\mathcal {B}$ be  the collection of all non separating essential disks of $V$.  Then there is an annu-disc system of $V$ in $\mathcal {A}\cup \mathcal {B}$,  says $\{A_{1}, A_{2},..., A_{l}, B_{1},...,B_{s}\}$  so that (1)   they are pairwise  disjoint; (2) their complement in $V$ is connected; (3) 
the complement of their boundary curves in $S$, denoted by$S_{l,s}$,  has genus $g(F)$,  $2[g(S)-g(F)]$ boundaries.  In this case, $l+s=g(S)-g(F)$.   

For any annu-disc system  $\{A_{1}, A_{2},..., A_{l}, B_{1},...,B_{s}\}$ of $V$,  either $D$ is disjoint from them or they intersect nontrivially.  Since $V$ is irreducible,  it is assumed that there is no cycle in their intersection.   In the later case,   their intersection consists of finitely many arcs. Then there is an outermost disk in $D$, which is bounded by 
some arc $\gamma\subset \partial D$ and some  arc $a_{1}$ in their intersection.  We call an annu-disc system  $\{A_{1}, A_{2},..., A_{l}, B_{1},...,B_{s}\}$ is tamed for $D$ if there is a component  $\gamma\subset \partial D\cap S_{l,s} $ so that (1) $\gamma$ lies in an outermost disk in $D$; (2) it is strongly essential in $S_{l,s}$ . Otherwise, it is   untamed.

If a given annu-disc system is untamed for $D$,  then  there is some  outermost disk and an arc $\gamma\subset \partial D \cap S_{l,s}$  so that $\pi_{S_{l,s}}(\gamma)$ bounds a disk in $V$. Then doing  a boundary compression along this outermost disk on $A_{i}$ (resp. $B_{j}$ ) produces a new non separating spanning annulus $A^{1}_{i}$ (resp. $B^{1}_{j}$).  It is known that $A^{1}_{i}$ shares the same boundary curve in $\partial_{-}V$ with $A_{j}$. Then there is a new annu-disc system $\{A_{1},..., A^{1}_{i},...,A_{l}, B_{1},...,B_{s}\}$ (resp. $\{A_{1},...A_{l}, B_{1},..B^{1}_{j},..B_{s}\}$). It is not hard to see that the intersection number between $D$ and the new annu-disc system is  less than before.  So we cyclically do this operation until  this annu-disc system is transformed into a tamed annu-disc system.

In all, for the essential disk $D$,  there is a tamed annu-disc system for it.  To  find a tamed disc system for $D$ in $V$,  there are some surgeries introduced on this tamed annu-disc system  \[\{A_{1}, A_{2},..., A_{l}, B_{1},...,B_{s}\}.\]  

For two spanning annuli  $A_{1}$ and $A_{2}$, which lie in the same component of $\partial_{-}V$,  there is an  arc  $a_{1,2}$ in $\partial _{-}V$ connecting them, whose interior is disjoint from this tamed annu-disc system.  So the I-bundle $a_{1,2}\times I$ connects $A_{1}$ and $A_{2}$ disjoint from this tamed annu-disc system.  If $\gamma$ is disjoint from $a_{1,2}\times I$,   then cutting the complement of this annu-disc system along  it produces a 3-manifold $V_{a_{1,2}}$. So the subsurface $S_{l,s}$ is cut into an essential subsurface $S_{l,s,a_{1,2}}$, see Figure 3.2.

\begin{figure}[!htbp]
\label{figure3.2}
\includegraphics[totalheight=4cm]{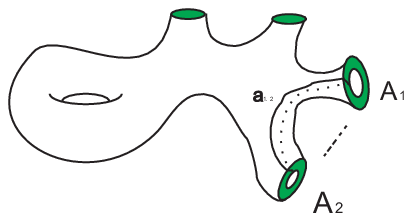}
\begin{center}
Figure 3.2.
\end{center}
\end{figure}

It is not hard to see that $\gamma\cap S_{l,s,a_{1,2}}$ is also  strongly essential in  $S_{l,s,a_{1,2}}$.  

Let $D(\gamma)$ be the disk bounded by  $\gamma$ and some arc in this annu-disc system in $V(r)$.
If  $\gamma$ intersects $a_{1,2}\times  I\cap S_{l,s}$ nontrivially, then  $D(\gamma)$ intersects  this I-bundle some arcs up to isotopy, where all these arcs have their ends in $a_{1,2}\times I \cap S_{l,s}$. Then there is an outermost disk in $D(\gamma)$ bounded by $\gamma_{1}$ and some arc in $a_{1,2}\times I$.  If this outermost disk is also in $V$,  then doing a boundary compression on $a_{1,2}\times I$ along it produces a I-bundle and a disk.  Here we still use $a_{1,2}\times I$ representing this  new I-bundle.  It is said that this new I-bundle has less intersection number with $D(\gamma)$.  Cyclically doing this operation until either $\gamma$ is disjoint from the resulted I-bundle or there is an outermost disk bounded by $\gamma_{1}$ and some arc in $D(\gamma)$ is in $V(r)$ but not in $V$.  In all, the strongly essential arc is denoted by $\gamma_{1}$. In this case $\gamma_{1}$ is strongly essential and bounds an essential disk in $V(r)$ not in $V$ with some arc in $\partial S_{l,s,a_{1,2}}$.   Let $A_{1,2}$ be the band sum of $A_{1}$ and $A_{2}$ along $a_{1,2}\times I$.  Then it  is also a spanning annulus in $V$. Furthermore, there is a collection of annuli and disks $\{A_{1,2}, A_{3},...A_{l}, B_{1},..., B_{s}\}$ so that $\gamma_{1}$ lies in the complement of it in $V$.

Cyclically doing the above operation until  there is no spanning annulus in this tamed annu-disc system.  Then at last it is transformed into a tamed disc system for $D$,  i.e., a collection of essential disk in $V$.  It is known that one component  of their complement in $V$ is $F\times I$. Let $S_{F}\subset S$ be the  component of their boundary's complement in $S$,  which lies in $F\times I$. Then  there is an arc $\gamma^{*} $  of $\partial D \cap S_{F}$  so that  it not only lies in an outermost disk in $D$ but also  is strongly essential in $S_{F}$. 

We summarize the above argument into a lemma as follows:
\begin{lem}
\label{lemma3.1}
For any essential disk $D$ in $V(r)$ but not in $V$, there are finitely many essential disks $\{B_{1},...,B_{s}\}$ of $V$ so that (1) one component of their complement in $V$ is $F\times I$; (2) the other components are  some closed surfaces I-bundles if possible; (3)  $D$ is in a SE-position with respect to $\cup_{i=1}^{s}B_{i}$, i.e.,  for some component $\gamma^{*}\subset \partial D\cap S_{F}$, $\pi_{S_{F}}(\gamma^{*})$ not only is strongly essential in $S_{F}$ but also bounds an essential disk in $V(r)$.

\end{lem}

\section{An upper bound on distance degenerate handle additions }
\label{sec4}
Suppose $V\cup_{S}W$ has distance $m\geq 3$.  An essential separating disk $B\subset V$ is  called a $F$-disk if  one component of $\overline{V-B}$ is  $F\times I$.   Let \[\mathcal {N}=\min\{ \mid B\cap E\mid  \vert E:~ an~ essential ~disk ~in ~W;~B: ~a ~F-disk\}.\]  Then there are an essential disk $E\subset W$ and a $F$-disk $B\subset V$ so that $\mathcal {N}=\mid B\cap E\mid$. 

Since cutting $V$  along a $F$-disk $B$ produces a  closed surface I-bundle $F\times I$,  there is a component $S_{1} \subset S-\partial B$  in $F\times I$.  It has been  discussed that $\partial E$ intersects $S_{1}$ nontrivially.  So the subsurface projection $\pi_{S_{1}}(\partial E)$ is an essential simple closed curve in $S_{1}$.  Since $S_{1}$ lies in $F\times I$, there is an essential simple closed curve $c\subset F$ so that the  union of $c$ and $\pi_{S_{1}}(\partial E)$ bound a spanning annulus in $V$. Moreover, $c$ is unique up to isotopy.  Therefore to get an upper bound for all distance degenerate curves in $F$,  it is sufficient to give an upper bound of  distances between all these degenerate curves  and $c$ in $\mathcal {C}(F)$.  More precisely,  let $r\subset F$ be a distance degenerate curve  for  $V\cup_{S}W$. It is known that $d_{\mathcal{C}(F)}(r, c)\leq d_{\mathcal {C}(F)}(r, \gamma_{l})+d_{\mathcal {C}(F)}(\gamma_{l}, b)+d_{\mathcal{C}(F)}(b,c)$. We will give an upper bound of $d_{\mathcal {C}(F)}(r, \gamma_{l})$ in Subsection 4.1, an upper bound of $d_{\mathcal{C}(F)}(b,c)$  in Subsection 4.2 and an upper bound of $d_{\mathcal {C}(F)}(\gamma_{l}, b)$ in Subsection 4.3. Then they together give an upper bound of $d_{\mathcal{C}(F)}(r, c)$.

\subsection{$d_{\mathcal {C}(F)}(r, \gamma_{l})\leq 2l\log_{2}{4[g(S)-g(F)]}+l+1$}
Since $r\subset F$ is a  distance degenerate curve for  $V\cup_{S}W$,  $d(V(r), W)=l\leq m-1$.  By  the definition of a Heegaard distance,  there is a collection of finitely many essential simple closed curves on $S$, says $\{\alpha_{0},...,\alpha_{l}\}$, so that  (I) $\alpha_{0}$ (resp. $\alpha_{l}$)  bounds an essential disk  $D_{0}$ (resp. $E_{l}$ ) in $V(r)$ (resp. $W$);  (II) for any $ 1\leq i\leq l$,  $\alpha_{i}$ is disjoint from $\alpha_{i-1}$.  Here $D_{0}$ is an essential disk in $V(r)$ but not in $V$. For if not,  then  $d(V, W)\leq l<m$.  Then by Lemma \ref{lemma3.1}, there are finitely many  essential disks $\{B_{1}, ..., B_{s}\}$ in $V$ so that either $\alpha_{0}$ is disjoint from all  these  disks or there is a strongly essential outermost  arc $\gamma\subset \alpha_{0}$ so that $\gamma$ and one arc in $D_{0}\cap B_{i}$, for some $1\leq i \leq s$,  bounds an essential disk.  In the first case,  $\alpha_{0}$ is strongly essential in $S_{F}$.  In the later case,  $\gamma$ is strongly essential in $S_{F}$ and so is  $\pi_{S_{F}}(\gamma)$.  So in both of these two cases, there is an essential disk  bounded by $\pi_{S_{F}}(\gamma)$ or $\alpha_{0}$ in  the component of $V(r)-\cup_{i=1}^{s} B_{i}$, says $V_{F}(r)$.  For simplicity,  they are both denoted by $\pi_{S_{F}}(\gamma)$.  

Since $V(r)$ is obtained from attaching a 2-handle along $r$ on $V$,   there is also an essential disk bounded by $r^{*}$ in $V_{F}(r)$, which bounds a spanning annulus $\mathcal {A}_{r}$ with $r$ in $V$.  So how  does $r^{*}$ intersect $\pi_{S_{F}}(\gamma)$ in $V_{F}(r)$?
 \begin{lem}
\label{lemma4.1}
\begin{enumerate}
\item If $r$ is separating in $F$, then for some choice of $r^{*}$, $\pi_{S_{F}}(\gamma)$  is isotopic to $r^{*}$ in $V_{F}(r)$;
\item If $r$ is non separating  in $F$,  then for some choice of $r^{*}$, $\pi_{S_{F}}(\gamma)$ is disjoint from $r^{*}$ in $V_{F}(r)$.
\end{enumerate}
\end{lem}
\begin{proof}
If $r$ is separating in $F$, then $V_{F}(r)$ contains only one essential disk $D_{r}$ up to isotopy. For if not,  then there is another essential disk in $V_{F}(r)$.  In this case,  either it is disjoint from the disk bounded by $r^{*}$ or it intersects $D_{r}$ nontrivially. But since $V_{F}(r)$ is homeomorphic to two closed surface I-bundles linked by a 1-handle,  it is impossible.  It is known that $\pi_{S_{F}}(\gamma)$  bounds an essential disk in $V_{F}(r)$. So $\pi_{S_{F}}(\gamma)$ is isotopic to $r^{*}$.

If $r$ is not separating in $F$, then $V_{F}(r)$ contains an  essential non separating disk $D_{r}$. Since $\pi_{S_{F}}(\gamma)$  also bounds an essential disk $D_{1}$  in $V_{F}(r)$, either $D_{1}$ is disjoint from $D_{r}$ or they intersects nontrivially.   In the former case, $\pi_{S_{F}}(\gamma)$  is disjoint from $r$.  In the later case,  it is assumed that there is no circle in their intersection. Then it consists of some arcs. Therefore there is an outermost disk in $D_{1}$, which is bounded by a component $\eta\subset \partial D_{1}$ and some arc in $D_{r}\cap D_{1}$. 

Cutting $\partial_{+} V_{F}(r)$ along $\partial D_{r}$ produces a compact surface $S_{V_{F}(r), \partial D_{r}}$, whose boundary curves are two copies of $\partial D_{r}$.  Then $\eta\subset  S_{V_{F}(r), \partial D_{r}}$ is strongly essential and  its two ends lie in the same copy of $\partial D_{r}$, says $\partial D_{r}^{1}$.   For if not, then it cuts out an annulus in $S_{V_{F}(r), \partial D_{r}}$,  which contains $\partial D_{r}^{2}$ as one boundary. 
It is not hard to see that for  any arc of $\pi_{S_{F}}(\gamma)\cap S_{V_{F}(r), \partial D_{r}} $,  if it has one end in $\partial D_{r}^{2}$, then the other end of it  is in $\partial D_{r}^{1}$, see Figure 4.1.

\begin{figure}[!htbp]
\label{figure4.1}
\includegraphics[totalheight=3cm]{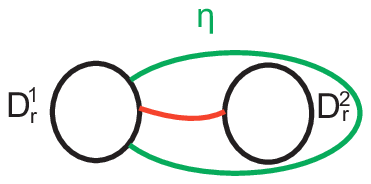}
\begin{center}
Figure 4.1.
\end{center}
\end{figure}

 However, the existence of $\eta$  shows that there are at least two more points in $\pi_{S_{F}}(\gamma)\cap \partial D_{r}^{1}$ than  in $\pi_{S_{F}}(\gamma)\cap \partial D_{r}^{2}$. So it contradicts the fact that $\partial D_{r}^{1}$  is isotopic to $\partial D_{r}^{2}$.  In this case,  $\pi_{ S_{V_{F}(r), \partial D_{r}}}(\eta)$ bounds an essential disk in $\overline {V_{F}(r)-D_{r}}$. But  $\overline {V_{F}(r)-D_{r}}$ is  a closed surface I-bundle and so contains no essential disk, a contradiction.
\end{proof}

So  $\pi_{S_{F}}(\gamma)$  is disjoint from $r^{*}$ after some isotopy. For simplicity,  $\pi_{S_{F}}(\gamma)$ is abbevirated by $\gamma_{0}^{*}$.

 For any $1\leq i \leq l$,  by the definition of Heegaard distance,  $\alpha_{i}$ intersects every essential disk of $V$ nontrivially.   
\begin{lem}
\label{lemma4.2}
Let $F$ be a genus at least one proper subsurface of $S$, i.e., essential but not isotopic, so that each component of $\partial F$ bounds an essential disk of $V$.  Then for any essential simple closed curve $\alpha\subset S$, if  $\alpha$ is disk busting in $V$, then there are $t\leq 2[g(S)-g(F)]$ arcs $\{a_{1}, ...,a_{t}\}$ of $\alpha\cap {F}$ so that one boundary curve of  $\overline{N[\partial {F}\cup(\cup_{i=1}^{t} a_{i})]}$ is strongly essential in $F$, says $\gamma^{*}$. 
\end{lem}	
\begin{proof} 
If $\partial F$ is connected, then every component of $\alpha\cap F$ is strongly essential.  Let $\gamma^{*}$ be an arbitrary one essential  curve of $\overline{N(\partial F\cup \alpha)}$ in $F$. So we assume that $\partial F$ is not connected.  If one component of $\alpha \cap F$ is strongly essential in $F$, then let 
$a_{1}$ be  the one. So $t=1$. Otherwise, none of $\alpha\cap F$ is strongly essential.  So there is at least one arc of $\alpha\cap F$ connecting two different boundary curves of $F$. For if not, then each arc of $\alpha\cap F$ cuts out a planar surface in $F$.  So there is an  essential boundary curve of $\overline {N[\partial F\cup (\alpha\cap F)]}$, denoted by $C$, so that it cuts out a planar surface in $F$.  Then $C$ bounds an essential disk in $V$. By the construction of $C$, it is disjoint from $\alpha$. Therefore $\alpha$ is not disk busting in $V$.  

Since $F$ has a finite genus and finitely many boundary curves, there are finitely many disjoint but nonisotopic  essential arcs in $\alpha\cap F$, says $\{a_{1}, ...,a_{t}\}$, so that each of them  connects two different boundary curves of $F$. Then  one boundary curve of  $\overline{N(\partial F\cup(\cup_{i=1}^{t} a_{i}))}$ is strongly essential in $F$.  For if not,  then for any choice of  these essential  arcs in $\alpha\cap F$,  there is an essential but not strongly essential  simple closed curve $C\subset F$ so that it  is disjoint from them.
For any two essential simple closed curves $C_{1}$ and $C_{2}$,  there  is a partial order $<$ defined. We say $C_{1}<C_{2}$ if $C_{1}$ is essential in the planar surface bounded by the union of $C_{2}$ and $\partial F$.  So for any sequence of  essential simple closed curves as above, there is a maximal one, denoted by  $C $.  Moreover $C$ is disjoint from $\alpha$. For if not,  then $C$ intersect $\alpha$ nontrivially. Since $C$ is  a union of some components $\alpha\cap F$ and some boundary arcs of $F$,  
$\alpha$ intersects these boundary arcs nontrivially. It means that  there is some arc $a$ of $\alpha\cap F$  which is not contained in the planar surface bounded by $C$ and $\partial F$. Then there is an essential but not strongly essential simple closed curve $C_{*}$ in $\partial \overline{N(\partial F\cup a\cup C)}$  so that  $C$ is essential in the planar surface bounded by $C_{*}$ and $\partial F$. Then it contradicts the maximality of $C$.  Since $C$ is a band sum of $\partial F$,  $C$ bounds an essential disk in $V$.  So $\alpha$ is not disk busting.

Since $F$ is an essential subsurface of $S$,  $\partial F$ has at most $2[g(S)-g(F)]$ components.  If $\gamma^{*}$ is the $\pi_{F}(\alpha)$,  then $t=1\leq 2[g(S)-g(F)]$.  Otherwise,  there are some pairwise  disjoint and nonistopic  arcs $\{a_{1},...a_{t}\}$  in $F$ so that $\gamma^{*}$ is a boundary component of  $\overline{N(\partial F\cup(\cup_{i=1}^{t} a_{i}))}$,  where $t$ is minimal. 
Since cutting $F$ along $a_{i}$ once reduces the number of $ \partial F$ by one,  the extreme case is that $F-\cup_{i=2}^{t} a_{i}$ is connected. Then  $t\leq 2[g(S)-g(F)]$.

\end{proof}

Since $\partial S_{F}$ consists of finitely many disks' boundary curves, for any $1\leq i\leq l$, by Lemma \ref{lemma4.2}, there is a strongly essential simple closed curve  $\gamma_{i}^{*}$ for $\alpha_{i}$ in $S_{F}$.  Since $S_{F}$ lies in $F\times I$, for any $0\leq i \leq l$, there is an essential simple closed curve $\gamma_{i}$ so that the union of $\gamma_{i}$ and $\gamma_{i}^{*}$ bound a spanning annulus $\mathcal {A}_{i}$ in $V$. 

 By Lemma \ref{lemma4.1}, $\gamma_{0}^{*}$ is disjoint from $r^{*}$ in $V_{F}(r)$.   Since $\gamma_{0}^{*}\cup \gamma_{0}$ (resp. $r^{*}\cup r$) bounds a spanning annulus $\mathcal {A}_{0}$ (resp. $\mathcal {A}_{r}$), the intersection number $\gamma_{0}\cap r$ is not larger than $\gamma_{0}^{*}\cap r^{*}$.  So $\gamma_{0}$ is disjoint from $r$ in $F$.   For  any $1\leq i\leq l$, since $\alpha_{i}\cap \alpha_{i-1}\neq \emptyset$, by Lemma \ref{lemma4.2}, $\gamma_{i}^{*}$  intersects  $\gamma_{i-1}^{*}$ in at most $2[g(S)-g(F)]$ points in $S_{F}$ and therefore $\gamma_{i}$ intersects $\gamma_{i-1}$ in at most $2[g(S)-g(F)]$ points in $F$.  By Lemma \ref{lemma2.0},  $d_{\mathcal {C}(F)}(\gamma_{i}, \gamma_{i-1})\leq 2\log_{2} {4[g(S)-g(F)]} +1$.  So $d_{\mathcal {C}(F)}(r, \gamma_{l})\leq 2l\log_{2}{4[g(S)-g(F)]}+l+1$. Thus to get an upper bound of the distance between $r$  and $c$ in $\mathcal {C}(F)$, it is sufficient to give an upper bound of  $d_{\mathcal {C}(F)}(c, \gamma_{l})$.

By Lemma \ref{lemma4.2},  there are at most $t\leq 2[g(S)-g(F)]$ arcs $\{a_{1},...a_{t}\}$ of $\partial E\cap S_{F}$ so that 
one boundary curve of $\partial \overline{N[\partial S_{F}\cup (\cup_{i=1}^{t} a_{i})]}$, says $\beta$,  is strongly essential in $S_{F}$. Moreover, there is an essential simple closed curve $b$ in $F$ so  that $\beta\cup b$ bound a spanning annulus in $V$. Thus to get an upper bound of $d_{\mathcal {C}(F)}(c, \gamma_{l})$,  it is enough to give these two estimations of $d_{\mathcal {C}(F)}(c, b)$ and $d_{\mathcal {C}(F)}(b, \gamma_{l})$.

\subsection{ $d_{\mathcal {C}(F)}(c, b)\leq2\log_{2}2\mathcal{N}+1$} 

Since $\partial D \cap \partial E=2\mathcal {N}$,  $\pi_{S_{1}}(\partial E)$ intersects $\partial E$ in at most $2\mathcal {N}$ points. So is $\pi_{S_{1}}(\partial E)\cap (\cup_{i=1}^{t}  a_{i}))$.  If $\pi_{S_{1}}(\partial E)$ is contained in $S_{F}$, then $\pi_{S_{1}}(\partial E) $ intersects $\beta$ in  at most $2\mathcal {N}$ points. But  when $\pi_{S_{1}}(\partial E)$ intersects $\partial S_{F}$ nontrivially,  it becomes  more subtler.  For explaining it,  there is a lemma introduced.
\begin{lem}
\label{lemma4.3}
There are two essential simple closed curve $\beta^{*}\subset S$, $c^{*}\subset S$ and $b\subset F$  so that 
\begin{enumerate}
\item  $c^{*}\cup c$ (resp. $\pi_{S_{1}}(\partial E)\cup c$) bounds  a spanning annulus in $V$;
\item  $\beta^{*}\cup b$ (resp. $\beta\cup b$) bounds a spanning annulus in $V$; 
\item  $\mid \beta^{*}\cap c^{*}\mid \leq 2\mathcal {N}$.
\end{enumerate}
\end{lem}
\begin{proof}
 By the construction of $\beta$,  it is a union of  some arcs of $\partial S_{F}$ and some arcs in the interior of $S_{F}$. The former arc is marked by $+$ while the later one is marked by $-$. Then $+$ arcs and $-$ arcs appear in $\beta$ alternatively.
 
It is possible that there are some more points in $\beta\cap \pi_{S_{1}}(\partial E)$ than in $\beta\cap \partial E$. In this case,   all these new points  belong to the intersection points between $+$ arcs and $\partial E$.  Therefore for removing all these new points, it is necessary to make some surgeries  on both  $\beta$ and $\pi_{S_{1}}(\partial  E)$. 

Since  $\pi_{S_{1}}(\partial E)\cup c$ bounds a spanning annulus $\mathcal {A}$  in $V$,  by the standard innermost circle and outermost disk argument,  there is an outermost disk in $\mathcal {A}$ which is bounded by the union of an arc in $e_{i,1}\subset B_{i}$ and an arc $\eta_{B_{i}}\subset \partial E$.  Because this outermost disk is contained in $V$,  $\eta_{B_{i}}$ is not strongly essential in $S_{F}$. So $\eta_{B_{i}}$ cuts $S_{F}$ out a planar surface, denoted by $S_{\eta_{B_{i}}}$.

 If $\beta \cap \eta_{B_{i}}\neq \emptyset$, then for any point $p\in A$,  there is  a surgery on $\beta$ along $\eta_{B_{i}}$ so that $p$ is removed from $\beta\cap \pi_{S_{1}}(\partial E)$, denoted by $\beta^{1}$, see Figure  4.2 for example.
 \begin{figure}[!htbp]
\label{figure4.2}
\includegraphics[totalheight=4cm]{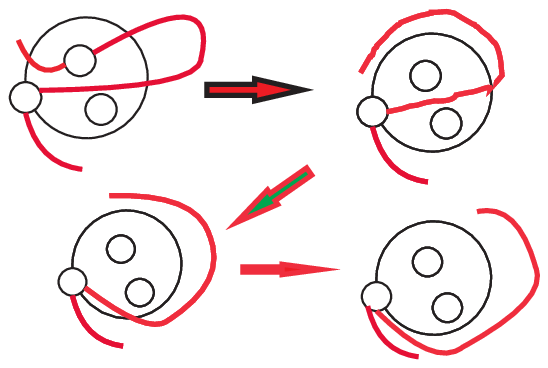}
\begin{center}
Figure 4.2. 
\end{center}
\end{figure}
 
 It is not hard to see that $\beta^{1}\cup b$ also bound a spanning annulus in $V$.  Moreover, there is no new point generated in this process.  If $\beta\cap \eta_{B_{i}}=\emptyset$,  then let $\beta^{1}=\beta$.  So  $\beta^{1}$ is  a union of some $+$ arcs and some $-$ arcs.  Then (1) $\beta^{1}\cup b$ bound a spanning annulus in $V$;  (2) $\beta^{1}$ is disjoint from $S_{\eta_{B_{i}}}$; (3) there are at most $2\mathcal {N}$ points  in intersection of   $-$ arcs between  $\beta^{1}$ and $\partial \mathcal {A}$. 

 Since $e_{i,1}$ cuts out  a disk in $B_{i}$,  there is an  outermost disk $B_{i,1}\subset B_{i}$ for the spanning  annulus $\mathcal {A}$. In this case, $\partial B_{i,1}$ consists of an arc in $B_{i}$ and an arc in  $S_{\eta_{B_{i}}}$.  Then doing a boundary compression along $B_{i,1}$ cuts $\mathcal {A}$ into a spanning annulus $\mathcal {A}_{1}$ and an essential disk in $V$.  In this process, since $S_{\eta_{B_{i}}}$ is disjoint from $\beta^{1}$, there is no new point generated in $\partial \mathcal{A}_{1}\cap \beta^{1}$. So $-$ arcs of $\beta^{1}$ intersects $\partial \mathcal{A}_{1}$ in at most $2\mathcal {N}$ points.
 
Let $c^{1}$ be $\mathcal {A}_{1}\cap S$.  If $c^{1}$ doesn't intersect $\partial S_{F}$ essentially,   then there is a bigon bounded by the union of $c^{1}$ and $\partial S_{F}$ in $S$. We assume that $\beta^{1}$ is disjoint from this bigon. For if there is a smaller bigon bounded by $\beta^{1}\cup c^{1}$, then we push $c^{1}$ over this smaller bigon so that it vanishes. In this process, there is no new point generated.  So we push $c^{1}$ over this bigon, bounded by $c^{1}\cup \partial S_{F}$,  so that it vanishes.  Also in this process,  there is no new  point  $\beta_{1}\cap c^{1}$ generated for $\beta^{1}\cup c^{1}$.  In all, ${-}$ arcs of $\beta^{1}$ intersects $c^{1}$ in at most $2\mathcal {N}$ points.

 After finitely many steps,  $c^{1}$ intersects $\cup_{i=1}^{s} \partial (B_{i})$ essentially and $c^{1}\cup c$ bounds a spanning annulus in $V$ too. Moreover $-$ arcs of $\beta^{1}$ intersects $c^{1}$ in at most $2\mathcal {N}$ points. 
 
 Cyclicly doing this operation until $c^{*}$ is disjoint from $\partial S_{F}$. Under this circumstance,   $-$ arcs of $\beta^{*}$ intersects $c^{*}$ in at most $2\mathcal {N}$ points. Since $c^{*}$ is disjoint from $+$ arcs of $\beta^{*}$,   $\beta^{*}$ intersects $c^{*}$ in at most $2\mathcal {N}$ points.  Moreover,  both $c^{*}\cup c$  and  $\beta^{*}\cup b$  bound  spanning annuli in $V$.
\end{proof}

By Lemma \ref{lemma4.3},  $\mid c^{*}\cap b^{*}\mid \leq 2\mathcal {N}$.  So $\mid c\cap b\mid\leq 2\mathcal{N}$. Therefore,  by Lemma \ref{lemma2.0}, $d_{\mathcal {C}(F)}(c, b)\leq 2\log_{2}2\mathcal{N}+1$.

\subsection{An upper bound of $d_{\mathcal {C}(F)}(b, \gamma_{l})$}

Since the Heegaard distance $d(V, W)$ is at least 3,  by Lemma {\ref{lemma2.1} and \ref{lemma2.2}},  for these two disks $E$ and $E_{\alpha_{l}}$ bounded by $\alpha_{l}$,  there is some connection between them. So do $\beta$ and $\gamma_{l}^{*}$.  Then there is  an upper bound of $d_{\mathcal {C}(F)}(b, \gamma_{l})$ obtained from it  as follows.   
\begin{lem}
\label{lemma4.4}
\begin{eqnarray*}
&&  g(S)-g(F)\geq 2,  d_{\mathcal {C}(F)}(b, \gamma_{l})<\\
&& {\frac{1}{2 \sqrt[]{6}}}[(\sqrt[]{6}-2)\frac{14\log_{2}2g(S)+7+(2-\sqrt[]{6})^{2g(S)-2g(F)-2}(50 \sqrt[]{6}-14\log_{2}{2g(S)}-117)}{\sqrt[]{6}-1}+\\
    && +  (\sqrt[]{6}+2)\frac{-14\log_{2}2g(S)-7+(2+\sqrt[]{6})^{2g(S)-2g(F)-2}(50 \sqrt[]{6}+14\log_{2}{2g(S)}+117)}{\sqrt[]{6}+1}];\\
  && g(S)-g(F)=1, d_{\mathcal {C}(F)}(b, \gamma_{l})\leq 52.
    \end{eqnarray*}

\end{lem}
\begin{proof}
There is a  mathematical induction  in this proof. 

{\bf (I)}  $S_{F}$ has only one boundary component. 

Since $d(V,W)\geq 3$,  $W$  is neither  a product I-bundle of $S_{F}$ nor a twisted I-bundle of $S_{F}$. 
By Lemma \ref{lemma2.1}, $d_{\mathcal {C} (S_{F})}(\beta, \gamma_{l}^{*})\leq 12$.  It is assumed that $S_{F}$ has only one boundary componnet. Then every essential  simple closed curve in $S_{F}$ is also strongly essential in it. So  $d_{\mathcal {C}(F)}(b, \gamma_{l})\leq 12$. 

{\bf (II)}  $S_{F}$ has exactly two boundary components.

 Since $d(V,W)\geq 3$,   $N> 4$.  For if not, then there is an essential disk $E_{0}\subset W$ so that $\partial E_{0}\cap \partial S_{F}$ has 2 or 4 points. If $\partial E_{0}\cap \partial S_{F}$ has 2 points, then $d(V,W)\leq 1$.  If $\partial E_{0}\cap  \partial S_{F}$ has  4 points, either all these 4 points lies in a same component of $\partial S_{F}$ or  there are 2 intersecting points  for each boundary component of $S_{F}$ individually.  But in both of these two cases, $d(V,W)\leq 2$. 
 
By Lemma \ref{lemma2.2},  there is an essential disk $E_{1}\subset W$ so that (1) it intersects $\partial S_{F}$ minimally;  (2)  a component $e_{1,1}$  (resp. $e_{1,2}$ ) of $\partial E_{1}\cap S_{F}$ is disjoint  from a component $e$ (resp. $e_{l}$) of $\partial E\cap S_{F}$ (resp. $\partial E_{l}\cap S_{F}$).
 
Let's firstly consider the case that all of these four arcs are strongly essential in $S_{F}$.  Then each one of  $\{\pi_{S_{F}}(e), \pi_{S_{F}}(e_{1,1}), \pi_{S_{F}}(e_{1,2}), \pi_{S_{F}}(e_{l})\}$ is strongly essential in $S_{F}$. For $\pi_{S_{F}}(e)$, there is an essential simple closed curve in $F$  so that the union of them bounds a spanning annulus in $V$.  In order to not introduce too many labels, this essential simple closed curve  in $F$  for $\pi_{S_{F}}(e)$ is still denoted by itself. So do the left three curves.  

It is not hard to see that 

\begin{eqnarray*}
\mid \pi_{S_{F}}(e)\cap \pi_{S_{F}}(e_{1,1}) \mid \leq 1;\\
\mid \pi_{S_{F}}(e_{1,1})\cap  \pi_{S_{F}}(e_{1,2}) \mid\leq 1;\\
\mid \pi_{S_{F}}(e_{1,2})\cap  \pi_{S_{F}}(e_{l}) \mid \leq 1.
\end{eqnarray*}
Then
\begin{eqnarray*}
d_{\mathcal {C} (F)}(\pi_{S_{F}}(e), \pi_{S_{F}}(e_{1,1}))\leq 2;\\
d_{\mathcal {C} (F)}(\pi_{S_{F}}(e_{1,1}), \pi_{S_{F}}(e_{1,2}))\leq 2;\\
d_{\mathcal {C} (F)}(\pi_{S_{F}}(e_{1,2}), \pi_{S_{F}}(e_{l}))\leq 2. 
\end{eqnarray*}
So 
\[d_{\mathcal {C} (F)}(\pi_{S_{F}}(e), \pi_{S_{F}}(e_{l}))\leq 6.\]

Since $\beta$ is a union of at most two components of $\partial E\cap S_{F}$ and some arcs of $\partial S_{F}$,  $\beta$ intersects $\pi_{S_{F}}(e)$ in at most one point. 
It is known that the union of $\beta$ and $b$ bounds a spanning annulus in $V$. So  $b$ intersects $\pi_{S_{F}}(e)$ in at most one point in $F$ up to isotopy. Then $d_{\mathcal {C}(F)}(b, \pi_{S_{F}}(e))\leq 2$.  Similarly, $d_{\mathcal {C}(F)}(\pi_{S_{F}}(e_{l}),\gamma_{l})\leq 2$.  Hence  $d_{\mathcal {C}(F)}(b, \gamma_{l})\leq 10$. 

The worst scenario is that none of $\{\partial E_{1}\cap S_{F}, \partial E_{1}\cap S_{F} , \partial E_{l}\cap S_{F}\}$ is strongly essential in $S_{F}$  while none of any  two arcs of$\{e, e_{1,1}, e_{1,2}, e_{l}\}$ is isotopic.   Under this circumstance, 
each component of $\{\partial E_{1}\cap S_{F}, \partial E_{1}\cap S_{F} , \partial E_{l}\cap S_{F}\}$ has its two ends in different boundary components of $S_{F}$.  For if not, let's consider $\partial E\cap S_{F}$ for example. Then there is one arc of $\partial E\cap S_{F}$ so that it cuts out a planar surface of $S_{F}$ and  a subsurface $S_{F, E}$, where $\partial S_{F,E}$ is connected.  It is known that $\partial E\cap \partial S_{F,E}$ is not an empty set.  Then for each component of $\partial E\cap S_{F,E}$,  it is a sub-arc of some component of $\partial E\cap S_{F,E}$, see Figure 4.3. It means that there is a strongly essential  arc of $\partial E\cap S_{F}$ in $S_{F}$.

\begin{figure}[!htbp]
\label{figure4.3}
\includegraphics[totalheight=2.5cm]{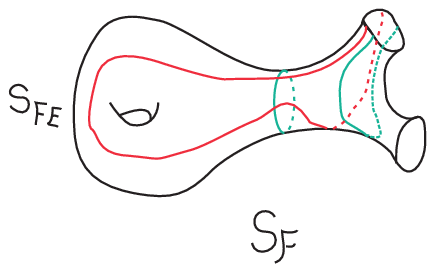}
\begin{center}
Figure 4.3.
\end{center}
\end{figure}

\begin{clm}
\label{claim4.0}
In the worst scenario, $d_{\mathcal  {C}(F)}(b, \gamma_{l})\leq 52$.
\end{clm}
\begin{proof}

Let  $S_{F,e}$ be the subsurface obtained from cutting $S_{F}$ along $e$. Then $\partial S_{F,e}$ is connected.  Since $\beta$ is the union of two components of $\partial E\cap S_{F}$ and two sub-arcs of $\partial S_{F}$,  $\beta$ intersects  $\pi_{S_{F,e}}(\partial E)$  in at most 2 points, see Figure 4.4.  By Lemma \ref{lemma4.3}, $d_{\mathcal {C}(F)}(b, \pi_{S_{F,e}}(\partial E))\leq 2$.

\begin{figure}[!htbp]
\label{figure4.4}
\includegraphics[totalheight=3cm]{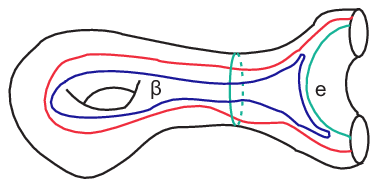}
\begin{center}
Figure 4.4.
\end{center}
\end{figure}

   Since $e_{1,1}$ is not isotopic to $e$,  the union of  $e_{1,1}$, $e$ and two boundary sub-arcs is a strongly essential simple closed curve in $S_{F}$, see Figure 4.4 for example.  Moreover, it is isotopic to $\pi_{S_{F,e}}(\partial E_{1})$. By Lemma \ref{lemma2.1},  since $W$ is not an I-bundle of $S_{F,e}$,  \[d_{\mathcal {C}(S_{F,e})}(\pi_{S_{F,e}}(\partial E), \pi_{S_{F,e}}(\partial E_{1}))\leq 12.\]  Since $\partial S_{F,e}$  is connected, for any essential simple closed curve  in $S_{F,e}$,  there is an essential simple closed curve in $F$ so that the union of them bound a spanning annulus in $V$.  In order to not introduce too many symbols, from now on, if there is no further notation,  for any strongly essential simple closed curve $C\subset S_{F}$, the corresponding essential simple closed curve  in $F$ is also represented by itself . So $d_{\mathcal {C}(F)}(b, \pi_{S_{F,e}}(\partial E_{1}))\leq 14$. 

Cutting $S_{F}$ along $e_{1,1}$ produces a subsurface $S_{F, e_{1,1}}$,  where  $\partial S_{F,e_{1,1}}$ is connected. Then by Lemma \ref{lemma2.1}, since $W$ is not an I-bundle of $S_{F, e_{1,1}}$,  $d_{\mathcal {C}(S_{F,e_{1,1}})}(\pi_{S_{F,e_{1,1}}}(\partial E), \pi_{S_{F,e_{1,1}}}(\partial E_{1}))\leq 12$.  Since  $\partial S_{F,e_{1,1}}$  is connected,  for any essential simple closed curve  in $S_{F,e_{1,1}}$,  there is an essential simple closed curve in $F$ so that the union of them bound a spanning annulus in $V$. So $d_{\mathcal {C}(F)}(\pi_{S_{F,e_{1,1}}}(\partial E_{1}), \pi_{S_{F,e_{1,1}}}(\partial E))\leq 12$.  

It is not hard to see that the union of $e$, $e_{1,1}$ and two boundary sub-arcs is  also isotopic to  $\pi_{S_{F,e_{1,1}}}(\partial E)$. Then $\pi_{S_{F,e_{1,1}}}(\partial E)$ is isotopic to $\pi_{S_{F,e}}(\partial E_{1})$. By the triangle inequality,  \[d_{\mathcal {C}(F)}(b, \pi_{S_{F,e_{1,1}}}(\partial E_{1}))\leq 26.\]  Similarly,   \[d_{\mathcal {C}(F)}(\gamma_{l}, \pi_{S_{F,e_{1,2}}}(\partial E_{1}))\leq 26.\]  Moreover, the union of $e_{1,1}$, $e_{1,2}$ and two boundary arcs is isotopic to not only   $\pi_{S_{F,e_{1,1}}}(\partial E_{1})$ but also $\pi_{S_{F,e_{1,2}}}(\partial E_{1})$.  Then by the triangle inequality again,
\[d_{\mathcal {C}(F)}(b, \gamma_{l})\leq 52.\]

\end{proof}

In  general, let's firstly consider the case that  $e_{1,1}$ is strongly essential in $S_{F}$.

\begin{clm}
\label{claim4.1} 
If $e_{1,1}$ is strongly essential in $S_{F}$,  then
$d_{\mathcal {C}(F)}(b, \gamma_{l})\leq 40.$
\end{clm} 
\begin{proof}
If $e$ is also strongly essential in $S_{F}$,  then $d_{\mathcal {C}(F)}(\pi_{S_{F}}(e), \pi_{S_{F}}(e_{1,1}))\leq 2$. Since there are at most two points in the intersection between $\pi_{S_{F}}(e)$ and $\beta$,  $b$ intersects $\pi_{S_{F}}(e)$ in at most two points in $F$. So $d_{\mathcal {C}(F)}(b, \pi_{S_{F}}(e_{1,1}))\leq 4$.

 If $e$ is not strongly essential in $S_{F}$,  then there is an essential subsurface $S_{F, e}$ of $S_{F}$ obtained from cutting $S_{F}$ along $e$. By Lemma \ref{lemma2.1}, since $W$ is not an I-bundle of $S_{F,e}$, \[d_{\mathcal {C}(S_{F,e})}(\pi_{S_{F,e}}(\partial E), \pi_{S_{F,e}}(\partial E_{1}))\leq 12.\]
  On one hand, $e_{1,1}\cap S_{F,e}$ is an essential arc of $\partial E_{1}\cap S_{F,e}$.  Then $\pi_{S_{F,e}}(e_{1,1})$ is isotopic to $\pi_{S_{F,e}}(\partial E_{1})$. Since $e_{1,1}$ is strongly essential,  $\partial e_{1,1}$ lies in a same boundary.  Therefore $\pi_{S_{F}}(e_{1,1})$ is isotopic to both  $\pi_{S_{F,e}}(e_{1,1})$ and  $\pi_{S_{F,e}}(\partial E_{1})$.  On the other hand, $\beta$ intersects $\pi_{S_{F,e}}(\partial E)$ in at most two points. So  $d_{\mathcal {C}(F)}(b, \pi_{S_{F,e}}(\partial E))\leq 2$. 
 By the triangle inequality,  $d_{\mathcal {C}(F)}(b, \pi_{S_{F,e}}(\partial E_{1}))\leq 14$. In all,
 \begin{equation}
 \label{equation4.6}
d_{\mathcal {C}(F)}(b, \pi_{S_{F}}(e_{1,1}))\leq 14.  
\end{equation}

{\bf Case 4.6.1.}  $e_{1,2}$ is strongly essential in $S_{F}$. 

Then by the above argument, $d_{\mathcal {C}(F)}(\gamma_{l}, \pi_{S_{F}}(e_{1,2}))\leq 14$. Since $e_{1,1}$ is disjoint from $e_{1,2}$,  $d_{\mathcal {C}(F)}(\pi_{S_{F}}(e_{1,1}), \pi_{S_{F}}(e_{1,2}))\leq 2$.  Then $d_{\mathcal {C}(F)}(b, \gamma_{l})\leq 30$.

{\bf Case 4.6.2.} $e_{1,2}$ is not strongly essential in $S_{F}$. 

 Then there is an essential subsurface $S_{F,e_{1,2}}\subset S_{F}$ obtained from cutting $S_{F}$ along $e_{1,2}$ so that $\pi_{S_{F}}(e_{1,1})$ is isotopic to $\pi_{S_{F,e_{1,2}}}(\partial E_{1})$. 

 {\bf Subcase 4.6.2.1.} $e_{l}$ is strongly essential in $S_{F}$.
 
Since $e_{1,2}$ is disjoint from $e_{l}$,  $\pi_{S_{F}}(e_{l})$ is isotopic to $\pi_{S_{F,e_{1,2}}}(\partial E_{l})$.  Since $W$ is not a I-bundle of $S_{F,e_{1,2}}$,  by Lemma \ref{lemma2.1}, \[d_{\mathcal {C}(S_{F,e_{1,2}})}(\pi_{S_{F,e_{1,2}}}(\partial E_{l}), \pi_{S_{F,e_{1,2}}}(\partial E_{1}))\leq 12.\]  So \[d_{\mathcal {C}(S_{F,e_{1,2}})}(\pi_{S_{F}}(e_{1,1}), \pi_{S_{F}}(e_{l}))\leq 12 .\] It is not hard to see that every essential simple closed curve in $S_{F, e_{1,2}}$ is strongly essential in $S_{F}$.  Therefore \[d_{\mathcal {C}(F)}(\pi_{S_{F}}(e_{1,1}), \pi_{S_{F}}(e_{l}))\leq 12.\] Since $\gamma_{l}^{*}$ intersects $\pi_{S_{F}}(e_{l})$ in at most one point,  $d_{\mathcal{C}(F)}(\pi_{S_{F}}(e_{l}), \gamma_{l})\leq 2$. By the triangle inequality, \[d_{\mathcal {C}(F)}(\pi_{S_{F}}(e_{1,1}), \gamma_{l})\leq 14. \] So \[d_{\mathcal {C}(F)}(b, \gamma_{l})\leq d_{\mathcal {C}(F)}(b, \pi_{S_{F}}(e_{1,1}))+d_{\mathcal {C}(F)}(\pi_{S_{F}}(e_{1,1}), \gamma_{l})\leq 28.\]

{\bf Subcase 4.6.2.2.}  $e_{l}$ is not strongly essential in $S_{F}$.

Then there is an essential subsurface $S_{F, e_{l}}$ obtained from cutting $S_{F}$ along $e_{l}$. Since $e_{1,2}$ is not strongly essential, 
either $e_{l}$ is isotopic to $e_{1,2}$ or not. In the first case, since $W$ is not an I-bundle of $S_{F,e_{l}}$, by Lemma \ref{lemma2.1}, $d_{\mathcal {C}(S_{F,e_{l}})}(\pi_{S_{F,e_{l}}}(\partial E_{1}), \pi_{S_{F,e_{l}}}(\partial E_{l}))\leq 12$. So $d_{\mathcal {C}(F)}(\pi_{S_{F,e_{l}}}(\partial E_{1}), \pi_{S_{F,e_{l}}}(\partial E_{l}))\leq 12$. By the construction of $\gamma_{l}$, $\gamma_{l}$ intersects $\pi_{S_{F,e_{l}}}(\partial E_{l})$ in at most two points. Then $d_{\mathcal {C}(F)}(\gamma_{l}, \pi_{S_{F,e_{l}}}(\partial E_{l}))\leq 2$.  So
\[d_{\mathcal {C}(F)}(\pi_{S_{F,e_{l}}}(\partial E_{1}), \gamma_{l})\leq 14.\]
Since $\pi_{S_{F}}(e_{1,1})$ is isotopic to $\pi_{S_{F,e_{1,2}}}(\partial E_{1})$,   $\pi_{S_{F}}(e_{1,1})$ is isotopic to $\pi_{S_{F,e_{l}}}(\partial E_{1})$.
Then \[d_{\mathcal {C}(F)}(\pi_{S_{F}}(e_{1,1}), \gamma_{l})\leq 14.\]
By Equation \ref{equation4.6}, 
\[d_{\mathcal {C}(F)}(b, \gamma_{l})\leq  d_{\mathcal {C}(F)}(b, \pi_{S_{F}}(e_{1,1}))+d_{\mathcal {C}(F)}(\pi_{S_{F}}(e_{1,1}), \gamma_{l})\leq 28.\]

The left case is that $e_{l}$ is not isotopic to $e_{1,2}$. Then one of them has its ends in two different boundary curves of $S_{F}$, says $e_{1,2}$.  For if not, then at least one of $\{e_{1,2}, e_{l}\}$ is strongly essential. So $\pi_{S_{F}}(e_{1,1})$ is isotopic to $\pi_{S_{F, e_{1,2}}}(\partial E_{1})$. If $e_{l}$  also has its two ends in two different boundary curves of $S_{F}$, then by the argument of  Claim \ref{claim4.0},  \[d_{\mathcal {C}(F)}(\gamma_{l}, \pi_{S_{F,e_{1,2}}}(\partial E_{1}))\leq 26.\]  Then \[d_{\mathcal {C}(F)}(\pi_{S_{F}}(e_{1,1}), \gamma_{l})\leq 26.\]
Since $d_{\mathcal {C}(F)}(b, \pi_{S_{F}}(e_{1,1}))\leq 14$, \[d_{\mathcal {C}(F)}(b, \gamma_{l})\leq 40.\] 
If $e_{l}$ has its two ends in a same boundary component of $S_{F}$, then it cuts out a subsurface $S_{F,e_{l}}$ containing no $e_{1,1}$, see Figure 4.5.

\begin{figure}[!htbp]
\label{figure4.5}
\includegraphics[totalheight=3cm]{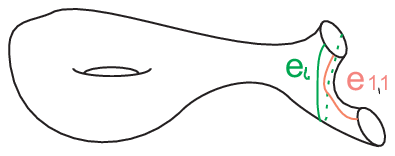}
\begin{center}
Figure 4.5.
\end{center}
\end{figure}

On one hand, $\pi_{S_{F}}(e_{1,1})$ (resp. $\gamma_{l}$)  is isotopic to  $\pi_{S_{F,e_{l}}}(\partial E_{1})$ (resp. $\pi_{S_{F,e_{l}}}(\partial E_{l})$). On the other hand, since $W$ is not an I-bundle of $S_{F,e_{l}}$, by Lemma \ref{lemma2.1}, \[d_{\mathcal {C}(S_{F,e_{l}})}(\pi_{S_{F,e_{l}}}(\partial E_{1}), \pi_{S_{F,e_{l}}}(\partial E_{l}))\leq 12.\]
So \[d_{\mathcal {C}(F)}(\pi_{S_{F}}(e_{1,1}), \gamma_{l})\leq 12.\] 

Since $d_{\mathcal {C}(F)}(b, \pi_{S_{F}}(e_{1,1}))\leq 14$,  
\[d_{\mathcal {C}(F)}(b, \gamma_{l})\leq 26.\]

Similarly,  for the case that $e_{l}$ has its ends in two different boundary components in $S_{F}$,  \[d_{\mathcal {C}(F)}(b, \gamma_{l})\leq 26.\]
\end{proof}

Similarly, if $e_{1,2}$ is strongly essential in $S_{F}$, then $d_{\mathcal {C}(F)}(b, \gamma_{l})\leq 40$. The left case is that  neither $e_{1,1}$ nor $e_{1,2}$ is strongly essential.  If $e_{1,1}$ is not isotopic to $e_{1,2}$, then either the union of $e_{1,1}$, $e_{1,2}$  and two sub-arcs of $\partial S_{F}$  is a closed strongly essential curve in $S_{F}$ or one of them, says $e_{1,1}$ for example, cuts out an essential subsurface $S_{F,e_{1,1}}$ containing no $e_{1,2}$. For the first case,  the closed strongly essential curve is isotopic to both $\pi_{S_{F,e_{1,2}}}(\partial E_{1})$ and $\pi_{S_{F,e_{1,1}}}(\partial E_{1})$. By the proof of Claim \ref{claim4.1}, \[d_{\mathcal {C}(F)}(\pi_{S_{F, e_{1,2}}}(\partial E_{1}), \gamma_{l})\leq 26.\]  Similarly, \[d_{\mathcal {C}(F)}(\pi_{S_{F, e_{1,1}}}(\partial E_{1}), b)\leq 26.\]
So \[d_{\mathcal {C}(F)}(b, \gamma_{l})\leq 52.\]

For the later case, without loss of generality, we assume that there is an essential subsurface $S_{F,e_{1,1}}\subset S_{F}$ of $e_{1,1}$  so that it  doesn't contain $e_{1,2}$.  Let $S_{F,e_{1,2}}$ be the surface obtained from cutting $S_{F}$ along $e_{1,2}$.

\begin{figure}[!htbp]
\label{figure4.6}
\includegraphics[totalheight=3cm]{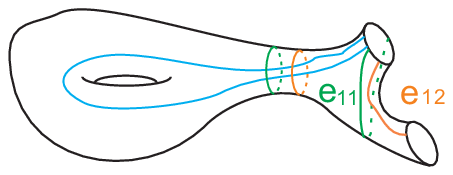}
\begin{center}
Figure 4.6.
\end{center}
\end{figure}

  Then it is not hard to see that $\pi_{S_{F,e_{1,2}}}(\partial E_{1})$ is isotopic to $\pi_{S_{F,e_{1,1}}}(\partial E_{1})$, see Figure 4.6. By the proof of Claim \ref{claim4.1}, \[d_{\mathcal {C}(F)}(\pi_{S_{F, e_{1,2}}}(\partial E_{1}), \gamma_{l})\leq 26.\]  Similarly, \[d_{\mathcal {C}(F)}(\pi_{S_{F, e_{1,1}}}(\partial E_{1}), b)\leq 26.\]
 So \[d_{\mathcal {C}(F)}(b, \gamma_{l})\leq 52.\]
 
If $e_{1,1}$ is isotopic to $e_{1,2}$,  by the similar argument, 
\begin{eqnarray*}
d_{\mathcal {C}(F)}(b, \pi_{S_{F,e_{1,1}}}(\partial E))&\leq &14;\\
d_{\mathcal {C}(F)}(\pi_{S_{F,e_{1,1}}}(\partial E), \pi_{S_{F,e_{1,2}}}(\partial E_{l}))&\leq& 12;\\
d_{\mathcal {C}(F)}(\pi_{S_{F,e_{1,2}}}(\partial E_{l}), \gamma_{l})&\leq &14.
\end{eqnarray*}
Therefore, \[d_{\mathcal {C}(F)}(b, \gamma_{l})\leq 40.\]

In all,  \[d_{\mathcal {C}(F)}(b, \gamma_{l})\leq 52.\]

{\bf (III)}  $ S_{F}$ has $n\geq 3$ boundary components. 

Since $d(V,W)\geq 3$,  by a similar argument in case 2, $N> 4$. By Lemma \ref{lemma2.2},  there is an essential disk $E_{1}\subset W$ so that (1) it intersects $\partial S_{F}$ minimally, (2) a  component $e_{1,1}$ of $\partial E_{1}\cap S_{F}$ is disjoint  from a component $e$ of $\partial E\cap S_{F}$. Similarly, there is a component $e_{1,2}$ of $\partial E_{1}\cap S_{F}$ disjoint from a component $e_{l}$ of $\partial E_{l}\cap S_{F}$. 

The most complicated case is that none of $\{e, e_{1,1}, e_{1,2}, e_{l}\}$  is strongly essential in $S_{F}$ while none of any two curves is isotopic. If $e$ has its two ends in a same boundary  component of $S_{F}$,  then cutting $S_{F}$ along it produces a nonplanar  subsurface $S_{F,e}$. If $e$ has its two ends in two different boundary components of $S_{F}$, then cutting $S_{F}$ along it produces a subsurface $S_{F,e}$.  

 Since $d(V,W)\geq 3$, by Lemma \ref{lemma4.2},  there is a strongly essential simple closed curve  $\beta_{1,e}$  in $S_{F,e}$ for  $\partial E$. Since $\beta$ is either disjoint from $e$ or intersects $e$ in at most  two points,  $\beta $ intersects $\beta_{1,e}$ in at most $2g(S)$ points.  For the essential disk $E_{1}$,  by Lemma \ref{lemma4.2} again,  there is a strongly essential simple closed curve $\theta$ (resp. $\theta_{1,e} $)for $E_{1}$ in $S_{F}$ ( resp. $S_{F,e}$).

 Let  $S_{F,e_{1,1}}$ be the nonplanar subsurface obtained from cutting $S_{F}$ along $e_{1,1}$. By Lemma \ref{lemma4.2},  there is a strongly essential simple closed curve $\beta_{1, e_{1,1}}$ (resp. $\theta_{1,e_{1,1}}$) for $\partial E$ (resp. $\partial E_{1}$) in $S_{F,e_{1,1}}$.

Let $S_{F, e,e_{1,1}}$ be the nonplanar subsurface obtained from cutting $S_{F}$ along $e $ and $e_{1,1}$. By Lemma \ref{lemma4.2},  there is a strongly essential simple closed curve $\beta_{1,e, e_{1,1}}$  (resp. $\theta_{1,e, e_{1,1}}$)  for  $\partial E$ (resp. $\partial E_{1}$) in $S_{F,e,e_{1,1}}$.
 It is not hard  to see that $\beta_{1,e_{1,1}}$  (resp. $\theta_{1,e}$ ) intersects $\beta_{1,e, e_{1,1}}$  (resp. $\theta_{1,e, e_{1,1}}$) in at most $2g(S)$ points. 

Let $S_{F, e_{1,2}}$ be  the nonplanar subsurface obtained from cuttting $S_{F}$ along $e_{1,2}$. By Lemma \ref{lemma4.2},  there is a strongly essential simple closed curve $\theta_{1, e_{1,2}}$ for $\partial E_{1}$.  Since $e_{1,1}$ is disjoint from $e_{1,2}$,  $\theta_{1,e_{1,1}}$ intersects $\theta_{1, e_{1,2}}$ in at most $2g(S)$ points.  By the similar argument, there is a strongly essential simple closed curve $\gamma_{l, e_{1,2}}$ for $\partial E_{l}$ in $S_{F, e_{1,2}}$.  

Let $S_{F,e_{l}}$ be the nonplanar subsurface obtained from cutting $S_{F}$  along $e_{l}$.  By Lemma \ref{lemma4.2},  there is a strongly essential simple closed curve $\theta_{1,e_{l}}$  (resp. $\gamma_{l,e_{l}}$)  for  $\partial E_{1}$ (resp. $\partial E_{l}$) in $S_{F,e_{l}}$.  It is not hard to see that $\gamma_{l, e_{l}}$ intersects $\gamma_{l}$ in at most $2g(S)$ points. 

Let $S_{F, e_{1,2}, e_{l}}$ be the nonplanar subsurface obtained from cutting $S_{F}$ along the union of $e_{1,2}$ and $e_{l}$.  By Lemma \ref{lemma4.2},  there is a strongly essential simple closed curve $\theta_{1,e_{1,2}, e_{l}}$  (resp. $\gamma_{l, e_{1,2}, e_{l}}$)  for  $\partial E_{1}$ (resp. $\partial E_{l}$) in $S_{F, e_{1,2}, e_{l}}$.   Since $e_{1,2}$ is disjoint from $e_{l}$,  
$\theta_{1,e_{l}}$ (resp. $\gamma_{1, e_{1,2}}$ ) intersects $\theta_{1,e_{1,2}, e_{l}}$ (resp. $\gamma_{l, e_{1,2}, e_{l}}$)  in at most $2g(S)$ points.  

Therefore, \begin{eqnarray*}
d_{\mathcal {C}(F)}(b, \gamma_{l}) &\leq& d_{\mathcal {C}(F)}(b, \beta_{1,e})+d_{\mathcal {C}(F)}(\beta_{1,e}, \theta_{1,e})+  d_{\mathcal{C}(F)}(\theta_{1,e}, \theta_{1,e,e_{1,1}})\\ &+&d_{\mathcal {C} (F)}(\theta_{1,e, e_{1,1}}, \beta_{1,e,e_{1,1}})+ d_{\mathcal {C}(F)}(\beta_{1,e,e_{1,1}}, \beta_{1,e_{1,1}})+d_{\mathcal {C}(F)}(\beta_{1,e_{1,1}}, \theta_{1,e_{1,1}})\\
&+& d_{\mathcal {C}(F)}(\theta_{1,e_{1,1}}, \theta_{1,e_{1,2}})+d_{\mathcal {C}(F)}(\theta_{1,e_{1,2}}, \gamma_{l, e_{1,2}})+d_{\mathcal {C}(F)}(\gamma_{l,e_{1,2}}, \gamma_{l,e_{1,2},e_{l}})\\ &+&  d_{\mathcal{C}(F)}(\gamma_{l,e_{1,2},e_{l}}, \theta_{1,e_{1,2},e_{l}})+d_{\mathcal {C} (F)}(\theta_{1,e_{1,2}, e_{l}}, \theta_{1,e_{l}})+ d_{\mathcal {C}(F)}(\theta_{1,e_{l}}, \gamma_{l, e_{l}})\\&+&d_{\mathcal {C}(F)}(\gamma_{l,e_{l}}, \gamma_{l})
\end{eqnarray*}

For any one of \[\{\beta\cap \beta_{1,e}, \theta_{1,e}\cap \theta_{1,e,e_{1,1}}, \beta_{1,e,e_{1,1}}\cap \beta_{1,e_{1,1}}, \theta_{1,e_{1,1}}\cap \theta_{1,e_{1,2}}, \gamma_{l,e_{1,2}}\cap 
\gamma_{l,e_{1,2}, e_{l}}, \theta_{1,e_{1,2}, e_{l}}\cap \theta_{1,e_{l}}, \gamma_{l,e_{l}}\cap \gamma_{l}\},\] it has at most $2g(S)$ points. It means that each one of \[\{b\cap \beta_{1,e}, \theta_{1,e}\cap \theta_{1,e,e_{1,1}}, \beta_{1,e,e_{1,1}}\cap \beta_{1,e_{1,1}}, \theta_{1,e_{1,1}}\cap \theta_{1,e_{1,2}}, \gamma_{l,e_{1,2}}\cap 
\gamma_{l,e_{1,2}, e_{l}}, \theta_{1,e_{1,2}, e_{l}}\cap \theta_{1,e_{l}}, \gamma_{l,e_{l}}\cap \gamma_{l}\},\] has at most $2g(S)$ points.
Then  by Lemma \ref{lemma2.0}, 
\begin{eqnarray*}
  d_{\mathcal {C}(F)}(b, \beta_{1,e})\leq 2\log_{2}2g(S)+1;\\
  d_{\mathcal{C}(F)}(\theta_{1,e}, \theta_{1,e,e_{1,1}})\leq 2\log_{2}2g(S)+1;\\
  d_{\mathcal {C}(F)}(\beta_{1,e,e_{1,1}}, \beta_{1,e_{1,1}})\leq 2\log2g(S)+1;\\
  d_{\mathcal {C}(F)}(\theta_{1,e_{1,1}}, \theta_{1,e_{1,2}})\leq 2\log2g(S)+1;\\
  d_{\mathcal {C}(F)}(\gamma_{l,e_{1,2}}, \gamma_{l,e_{1,2},e_{l}})\leq 2\log2g(S)+1;\\
  d_{\mathcal {C} (F)}(\theta_{1,e_{1,2}, e_{l}}, \theta_{1,e_{l}})\leq 2\log2g(S)+1;\\
  d_{\mathcal {C}(F)}(\gamma_{l,e_{l}}, \gamma_{l})\leq 2\log2g(S)+1.
\end{eqnarray*}
So, 
\begin{eqnarray*}
d_{\mathcal {C}(F)}(b, \gamma_{l}) &\leq& d_{\mathcal {C}(F)}(\beta_{1,e}, \theta_{1,e})+d_{\mathcal {C} (F)}(\theta_{1,e, e_{1,1}}, \beta_{1,e,e_{1,1}})+d_{\mathcal {C}(F)}(\beta_{1,e_{1,1}}, \theta_{1,e_{1,1}})\\
&+& d_{\mathcal {C}(F)}(\theta_{1,e_{1,2}}, \gamma_{l, e_{1,2}})+  d_{\mathcal{C}(F)}(\gamma_{l,e_{1,2},e_{l}}, \theta_{1,e_{1,2},e_{l}})+ d_{\mathcal {C}(F)}(\theta_{1,e_{l}}, \gamma_{l, e_{l}})\\&+&14\log2g(S)+7.
\end{eqnarray*}
 For any one of $\{S_{F, e}, S_{F, e_{1,1}}, S_{F, e_{1,2}}, S_{F, e_{l}}\}$, it has at most $n-1$ boundary curves; for any one of $\{S_{F, e, e_{1,1}}, S_{F, e_{1,2}, e_{l}}\}$,  it has at most $n-2$ boundary curves.   Thus to get an upper bound,  it is enough to consider the extreme case. Then there is a formula introduced.
\begin{eqnarray*}
 f(n)&=&4f(n-1)+2f(n-2)+14\log2g(S)+7, n\geq 3;\\
 f(2)&=&52, f(1)=12,\\
\end{eqnarray*}
where $\{f(n), n\subset N^{+}\}$ is a  Fibonacci series.  So there is a transformation of it as follows.
\begin{eqnarray*}
 f(n)+rf(n-1)+t&=&s(f(n-1)+rf(n-2)+t), n\geq 3;\\
 f(2)&=&52, f(1)=12.\\
\end{eqnarray*}

So 
\begin{eqnarray*}
s-r&=&4;\\
rs&=&2;\\
st-t&=&14\log2g(S)+7.
\end{eqnarray*}
Then there are two solutions,  which are 
\begin{eqnarray*}
s_{1}=&\sqrt []{6}+2, ~r_{1}=&\sqrt []{6}-2; \\
s_{2}=&2-\sqrt []{6},  ~r_{2}=&-2-\sqrt []{6}.
\end{eqnarray*}
Therfore
\begin{eqnarray*} 
&&(r_{1}-r_{2})f(n)+r_{1}t_{2}-r_{2}t_{1}=r_{1}s_{2}^{n-2}[f(2)+r_{2}f(1)+t_{2}]-r_{2}s_{1}^{n-2}[f(2)+r_{1}f(1)+t_{1}];\\
&&f(n)={\frac{1}{2 \sqrt[]{6}}}[(\sqrt[]{6}-2)\frac{14\log2g(S)+7+(2-\sqrt[]{6})^{n-2}(50 \sqrt[]{6}-14\log{2g(S)}-117)}{\sqrt[]{6}-1}+\\
  &&   +  (\sqrt[]{6}+2)\frac{-14\log2g(S)-7+(2+\sqrt[]{6})^{n-2}(50 \sqrt[]{6}+14\log{2g(S)}+117)}{\sqrt[]{6}+1}], n\geq 3;\\
   &&  f(2)=52, f(1)=12.
\end{eqnarray*}

Since $S_{F}$ has at most $2[g(S)-g(F)]$ boundary components,  $d_{\mathcal {C}(F)}(b, \gamma_{l})\leq f[2g(S)-2g(F)]$. Therefore 
\begin{eqnarray*}
&&d_{\mathcal {C}(F)}(b, \gamma_{l})\leq\\
 &&{\frac{1}{2 \sqrt[]{6}}}[(\sqrt[]{6}-2)\frac{14\log2g(S)+7+(2-\sqrt[]{6})^{2g(S)-2g(F)-2}(50 \sqrt[]{6}-14\log{2g(S)}-117)}{\sqrt[]{6}-1}+\\
   &&  +  (\sqrt[]{6}+2)\frac{-14\log2g(S)-7+(2+\sqrt[]{6})^{2g(S)-2g(F)-2}(50 \sqrt[]{6}+14\log{2g(S)}+117)}{\sqrt[]{6}+1}],\\
      &&g(S)-g(F)\geq 2;\\
     &&d_{\mathcal {C}(F)}(b, \gamma_{l})\leq 52, g(S)-g(F)=1.
    \end{eqnarray*}

For a general case,  by the similar argument,   $d_{\mathcal{C}(F)}(b, \gamma_{l})$ is not larger than the upper bound in the extreme case. 

\end{proof}

So 
\begin{eqnarray*}
d_{\mathcal {C}(F)}( r, c)\leq d_{\mathcal {C}(F)}(r, \gamma_{0})+d_{\mathcal {C}(F)}(\gamma_{0}, \gamma_{l})  +     d_{\mathcal {C}(F)}(\gamma_{l}, b)+d_{\mathcal {C}(F)}(b, c);\\
d_{\mathcal {C}(F)}(r, \gamma_{0})\leq1;\\
d_{\mathcal {C}(F)}(\gamma_{0}, \gamma_{l})\leq 2l\log2g(S)+l < 2m\log2g(S)+m;\\
d_{\mathcal {C}(F)}(b, c)\leq 2\log2\mathcal{N}+1;\\
d_{\mathcal {C}(F)}(b, \gamma_{l})\leq\\ {\frac{1}{2 \sqrt[]{6}}}[(\sqrt[]{6}-2)\frac{14\log2g(S)+7+(2-\sqrt[]{6})^{2g(S)-2g(F)-2}(50 \sqrt[]{6}-14\log{2g(S)}-117)}{\sqrt[]{6}-1}+\\
 +  (\sqrt[]{6}+2)\frac{-14\log2g(S)-7+(2+\sqrt[]{6})^{2g(S)-2g(F)-2}(50 \sqrt[]{6}+14\log{2g(S)}+117)}{\sqrt[]{6}+1}],\\g(S)-g(F)\geq 2;\\
 d_{\mathcal {C}(F)}(b, \gamma_{l})\leq 52, g(S)-g(F)=1.
 \end{eqnarray*}    
Then 
\begin{eqnarray*}
d_{\mathcal {C}(F)}( r, c)<2\log2\mathcal{N}+2m\log2g(S)+m+2+\\
+ {\frac{1}{2 \sqrt[]{6}}}[(\sqrt[]{6}-2)\frac{14\log2g(S)+7+(2-\sqrt[]{6})^{2g(S)-2g(F)-2}(50 \sqrt[]{6}-14\log{2g(S)}-117)}{\sqrt[]{6}-1}+\\ 
+  (\sqrt[]{6}+2)\frac{-14\log2g(S)-7+(2+\sqrt[]{6})^{2g(S)-2g(F)-2}(50 \sqrt[]{6}+14\log{2g(S)}+117)}{\sqrt[]{6}+1}],\\g(S)-g(F)\geq 2;\\
d_{\mathcal {C}(F)}( r, c)<2\log2\mathcal{N}+2m\log2g(S)+m+54, g(S)-g(F)=1.
\end{eqnarray*}

Let 

\begin{eqnarray*}
\mathcal {R}=\max\{ 2\log2\mathcal{N}+2m\log2g(S)+m+2+\\
+ {\frac{1}{2 \sqrt[]{6}}}[(\sqrt[]{6}-2)\frac{14\log2g(S)+7+(2-\sqrt[]{6})^{2g(S)-2g(F)-2}(50 \sqrt[]{6}-14\log{2g(S)}-117)}{\sqrt[]{6}-1}+\\ 
+  (\sqrt[]{6}+2)\frac{-14\log2g(S)-7+(2+\sqrt[]{6})^{2g(S)-2g(F)-2}(50 \sqrt[]{6}+14\log{2g(S)}+117)}{\sqrt[]{6}+1}], \\2\log2\mathcal{N}+2m\log2g(S)+m+54\}.
\end{eqnarray*}
\ Then $d_{\mathcal {C}(F)}( r, c)<\mathcal {R}$. Hence the proof of Theorem \ref{theorem1.1} ends.

 \subsection{The proof of Corollary 1.2}
By Scharlemann and Tomova' result \cite{ST},  every Heegaard splitting of $M$ has distance at most $\max\{d(V,W),2g(S)\}$. It is a result of Kobayashi and Rieck \cite{KR} (an extended  result of Schleimer \cite{schleimer1}  for compact 3-manfiolds) that  if $t$ is the number of tetrahedra and truncated tetrahedra in $M$,  then every genus  at least  $ 76t+26 $ Heegaard splitting  has distance at most 2.  Therefore for any  distance at least 3 Heegaard splitting of $M$, it has genus at most $76t+25$ and distance at most $\max\{d(V,W), 2g(S)\}$.  

By the generalized Waldhausen conjecture proved by Li \cite{li1,li2}, there are finitely many same genus but non isotopic Heegaard splittings for $M$.  So there are finitely many non isotopic distance at least 3 Heegaard splittings of $M$. Therefore there are a maximum $\mathcal {N}$ for all of these distance at least 3 Heegaard splittings and finitely many choices of $c$.  So there are a curve $c^{*}$ and a universal bound $\mathcal {R}^{*}$ in $\mathcal {C}(F)$  so that for any distance degenerate curve $r$ among all its distance at least 3 Heegaard splittings,
$d_{\mathcal {C}(F)}(c^{*}, r)<\mathcal {R}^{*}$.

In particular, if $M=E(K)$ for some knot $K\subset S^{3}$, then the meridian is a distance degenerate slope for any distance at least 3 Heegaard splitting.  So we write Corollary \ref{theorem1.2} as follows.
\begin{cor}
For any high distance knot $K\subset S^{3}$,  there is a $R_{K}$-ball  of the meridian in $\mathcal {C}[\partial E(K)]$   so that it contains all degenerate slopes of its all distance at least 3 Heegaard splittings.
\end{cor}
\begin{proof}
It is a direct result of the above argument.
\end{proof}

\vskip 3mm

{\it Yanqing Zou\\ {\tiny Department of Mathematics,  Dalian Minzu University \\ yanqing@dlnu.edu.cn; yanqing\_dut@163.com}}

\end{document}